\documentclass{amsart}
\newtheorem{theorem}{Theorem}[section]

\newtheorem{corollary}[theorem]{Corollary}
\theoremstyle{definition}

\theoremstyle{remark}

\numberwithin{equation}{section}

\newcommand{\C}{\mathbb{C}}
\newcommand{\D}{\mathbb{D}}
\renewcommand{\phi}{\varphi}
\newcommand{\R}{\mathbb{R}}
\newcommand{\W}{W_{\psi,\phi}}
\newcommand{\Z}{\mathbb{Z}}
\newcommand{\vstrut}{\rule{0in}{.15in}}
\begin{document}

 \title[Invertible weighted composition operators]{Invertible weighted composition operators}

%    optional.  temove any unused author tags.

%    author one information
\author{Paul S.\ Bourdon}
\address{Mathematics Department\\ Washington and Lee University\\Lexington, VA 24450}
\curraddr{}
\email{bourdonp@wlu.edu}
\thanks{}

%    author two information
%\author{}
%\address{}
%\curraddr{}
%\email{}
%\thanks{}

%    \subjclass is required.
\subjclass[2010]{Primary: 47B33; 30J99}

\date{}

\dedicatory{}

%    "Communicated by" -- provide editor's name; required.
%\commby{Richard Rochberg}

%    Abstract is required.
\begin{abstract}   Let $X$ be a set of analytic functions on the open unit disk $\D$, and let $\phi$ be an analytic function on $\D$ such that $\phi(\D)\subseteq \D$ and  $f\mapsto f\circ \phi$ takes $X$ into itself.  We present conditions on $X$ ensuring that if $f\mapsto f\circ\phi$ is invertible on $X$, then $\phi$ is an automorphism of $\D$, and we derive a similar result for  mappings of the form $f\mapsto \psi\cdot (f\circ \phi)$, where $\psi$ is some analytic function on $\D$.    We obtain as corollaries of this purely function-theoretic work,  new  results concerning invertibility of composition operators and weighted composition operators on Banach spaces of analytic functions such as $S^p$ and the weighted Hardy spaces $H^2(\beta)$.\end{abstract}

\maketitle

\section{Introduction}

Motivation for this paper derives from two sources: Theorems 1.6 and 2.15 of \cite{CMB},  which provide a condition ensuring that if a composition operator on a weighted Hardy space of the unit disk $\D$ is invertible, then its symbol is an automorphism of $\D$, and Theorem 2.0.1 of \cite{GGPP},  which characterizes invertible weighted composition operators on the classical Hardy space of the disk (and, by the same method, additional weighted Hardy space \cite[p.\ 860]{GGPP}).  These invertibility theorems are produced with the aid of reproducing kernels for the spaces in question.  We obtain here more general results, as corollaries of theorems on invertibility of composition operators and weighted composition operators on {\em sets} of analytic functions without linear or norm structure and hence without reproducing kernels.  Our work permits us to completely characterize invertibility of composition operators and weighted composition operators on automorphism-invariant functional Banach spaces such as $S^p$, which consists of analytic functions on $\D$ having derivatives in the Hardy space $H^p(\D)$.   We also show that if a composition operator $f\mapsto f\circ \phi $ or weighted composition operator  $f\mapsto \psi\cdot (f\circ\phi)$ on any weighted Hardy space $H^2(\beta)$ is invertible, then $\phi$ must be an automorphism of $\D$.   

 Throughout this paper $\psi$ and $\phi$  represent  analytic functions on $\D$, with $\phi$ having the additional property $\phi(\D)\subseteq \D$.  Thus, $\phi$ denotes an analytic selfmap of $\D$.   Let $X$ be a {\em set} of analytic functions on $\D$.  We emphasize that $X$ is not assumed to have linear or norm structure.  For example, $X$ might be  the set $X_{nz}$ of analytic functions on $\D$ that vanish at no point of $\D$.   We say the selfmap $\phi$ of $\D$ induces a composition operator $C_\phi$ on $X$ provided $C_\phi f:=f\circ \phi$ belongs to $X$ whenever $f\in X$. If $\phi$ and $\psi$ are such that $W_{\psi,\phi} f :=\psi\cdot (f\circ \phi)$ belongs to $X$ whenever $f\in X$, then we say $\psi$ and $\psi$ induce a weighted composition operator $W_{\psi, \phi}$ on $X$.   Observe that any analytic selfmap $\phi$ on $\D$ induces a composition operator on $X_{nz}$ and if $\psi$ is nonzero on $\D$, then $W_{\psi, \phi}$ will be a weighted composition operator on $X_{nz}$.  

If $X$ is a vector space, then any composition operator or weighted composition operator defined on $X$ will be linear.  Composition operators on normed linear spaces $X$ have been studied extensively (see, e.g., the texts \cite{CMB} and \cite{SB}) with issues such as boundedness, compactness, cyclicity, and spectral behavior receiving considerable attention.  Similar studies of weighted composition operators have been undertaken  (see, e.g. \cite{AFC}, \cite{BtJOT},  \cite{GalP}, \cite{GGC},  \cite{GG}, \cite{GGPP}, and  \cite{VM}).  In these studies of composition and weighted composition operators, the space $X$ in question is most often a weighted Hardy space. 

   A Hilbert space comprising functions analytic on $\D$ in which the polynomials are dense and the monomials $1$, $z$, $z^2$, \ldots, constitute an orthogonal set of nonzero vectors is  a {\it weighted Hardy space}.  Each weighted Hardy space is characterized by its {\it weight sequence} $\beta$ defined by  $\beta(j) = \|z^j\|$ for  $j\ge 0$.  The weighted Hardy space $H^2(\beta)$ consists of those functions $f$ analytic on $\D$ whose Maclaurin coefficients $(\hat{f}(j))$ satisfy
$$
\sum_{j=0}^\infty |\hat{f}(j)|^2\beta(j)^2 < \infty.
$$
The inner product of $H^2(\beta)$ is given by
$$
\langle f, g\rangle = \sum_{j=0}^\infty \hat{f}(j)\overline{\hat{g}(j)}\beta(j)^2.
$$
If $\beta(j) = 1$ for all $j$, then $H^2(\beta)$ is the classical Hardy space $H^2$ of the disk; the choices $\beta(j) = (j+1)^{-1/2}$ and $\beta(j) = (j+1)^{1/2}$ yield, respectively, the classical Bergman and Dirichlet spaces of the disk.  As is customary,  we make the normalizing assumption that  $\beta(0) = 1$.    It's not difficult to show that requiring functions in $H^2(\beta)$ to be analytic on $\D$ is equivalent to requiring 
  $\liminf \beta(j)^{1/j} \ge 1$ (see, e.g., exercise 2.1.10 of \cite{CMB}).

 Theorems 1.6 and 2.15 of \cite{CMB} combine to show that if $C_\phi$ is a bounded invertible operator on $H^2(\beta)$ and 
\begin{equation}\label{CFI}
\sum_{n=0}^\infty \frac{1}{\beta(j)^2} = \infty
\end{equation}
then $\phi$ must be an automorphism of $\D$.  This condition is generalized  in \cite[Theorem 2]{MP}, where $\sum_{n=0}^\infty n^{2k}/\beta(n)^2 =\infty$ for some $k\ge 0$ is shown to be sufficient to imply that any Fredholm composition operator $C_\phi$ must have its symbol $\phi$ be an automorphism. Here we show that for {\em every} weighted Hardy space $H^2(\beta)$, if $C_\phi: H^2(\beta)\rightarrow H^2(\beta)$ is invertible, then $\phi$ must be an automorphism of $\D$---see Theorem~\ref{WHST} below.

 The methods used to prove Theorem 2.0.1  of   \cite{GGPP} show $\phi$ to be an automorphism  if $W_{\psi,\phi}$ is  bounded and invertible on a range of weighted Hardy spaces including the classical Hardy space $H^2(\D$) and the standard-weight Bergman spaces $L^2_a(\D, (1-|z|^2)^\delta\ dA/\pi)$, $\delta > -1$.  Here, we obtain the same result for all weighed Hardy spaces (Theorem~\ref{WHST}).     These applications of our work to weighted Hardy spaces are presented in Section 3.  Also, in Section 3, we discuss how, for automorphism-invariant spaces like $S^p$, our work completely characterizes invertible composition operators and weighted composition operators.

\section{Invertibility of $C_\phi$ and $W_{\psi, \phi}$ on sets of analytic functions}

The goal of this section is to prove the following theorem as well as a similar one for weighted composition operators. Recall that $\phi$ is an analytic selfmap of the open unit disk $\D$. 

\begin{theorem}\label{MT1} Suppose that  $X$ is a set of analytic functions on $\D$ such that (i) $X$ is invariant under composition with $\phi$, (ii) $X$ contains a univalent function, (iii) $X$ contains a nonconstant function analytic on a neighborhood of the closed disk, and (iv) there is a dense subset $S$ of the unit circle such that for each point in $S$ there is function in $X$ that does not extend analytically to a neighborhood of that point.  If $C_\phi$ is invertible on $X$, then $\phi$ is an automorphism of the disk $\D$.
\end{theorem}

Note that any set $X$ containing $f(z) = z$, such as any weighted Hardy space, immediately satisfies hypotheses (ii) and (iii) of the preceding theorem.   Before proving Theorem~\ref{MT1}, we point out that some additional hypotheses on $X$ such as those provided in its statement are needed to ensure that if $C_\phi$ is invertible on $X$, then $\phi$ is an automorphism of $\D$.  Assuming simply that $X$ contains a nonconstant function is not sufficient.     For instance, if $X$ is the set of entire functions, then the non-automorphism $\phi(z) = z/2$ will induce an invertible composition operator on $X$.  

 To show that a hypothesis like (ii) is needed, we rely on eigenfunctions for composition operators.   Whenever $\phi$ fixes a point $\omega$ on the unit circle and $\phi'(\omega)< 1$, $C_\phi$ will have nonconstant eigenfunctions (see, e.g., Lemma 7.24 of \cite{CMB}). Also if  $\phi$ is a non-automorphic selfmap of $\D$ satisfying $\phi(0) = 0$ and $\phi'(0) \ne 0$; then $\phi$ has a {\it Koenigs eigenfunction}   $\sigma$ (see, e.g., \cite[\S 6.1]{SB} ), which is a holomorphic function on $\D$ satisfying 
$$
C_\phi \sigma = \phi'(0)\sigma.
$$
Let $\{(f_\alpha, \lambda_\alpha): \alpha\in A\}$, be an indexed collection of eigenfunction-eigenvalue pairs for $C_\phi$.  Assume that $\phi$ is not constant so that no eigenvalue $\lambda_a$ is zero.  Define
  \begin{equation}\label{KES}
X = \bigcup_{\alpha \in A} \{\lambda_\alpha^kf_\alpha:k\in \Z\}
\end{equation}
and observe that $C_\phi$ is invertible on $X$.  
Note that choosing $\phi$ to be non-univalent makes any of its eigenfunctions non-univalent.

\begin{proof}[Proof of Theorem~\ref{MT1}] Suppose that $C_\phi$ is invertible on $X$.   Because $X$ contains a univalent function $g$ and $C_\phi$ is invertible, there is a function $q\in X$ such that $q\circ \phi = g$.  Thus $g$ would identify any two points identified by $\phi$, and thus $\phi$ must be univalent.   

 Because $X$ contains a nonconstant  function $h$ that is analytic on the closed disk and $C_\phi$ is invertible, there is a function $f\in X$ such that $f\circ \phi = h$.  Suppose, in order to obtain a contradiction, that $\phi$ has radial limit of modulus less than $1$ on a subset $E$ of $\partial \D$ having positive Lebesgue measure.  Because $E$ has positive measure,  there is a positive number $t$ less than $1$ such that  measure of the set $T:= \{\zeta\in \partial \D:  |\phi(\zeta)| < t\}$ is also positive.  Note that the set $\phi(T)$ cannot be finite for then $\phi$ would map a subset of $T$ having positive measure to a single point, making $\phi$ constant (contradicting its univalence). Because $f\circ \phi$ equals $h$, a nonconstant function, $f$ must also be nonconstant.  Thus its derivative must be nonzero at a point $\phi(\zeta_0)$ of $\phi(T)$. Thus there is a disk $D_1$ centered at $\phi(\zeta_0)$ and contained in $\D$ on which $f$ is invertible with inverse $f^{-1}$.   Recalling that $h$ is analytic on the closed disk and that $f\circ \phi = h$, we must have $f(\phi(\zeta_0)) = h(\zeta_0)$.  Thus, $f(D_1)$ is a neighborhood of $h(\zeta_0)$ and it follows that $f^{-1}\circ h$ is analytic on an open disk $D_2$ centered at $\zeta_0$.  However for $r\in [0,1)$ sufficiently close to $1$, $r\zeta$ will be in $D_2$ and $\phi(r\zeta)$ will be in $D_1$; thus, for such $r$,
$\phi(r\zeta) = (f^{-1}\circ h)(r\zeta)$.  It follows that $f^{-1}\circ h$ is an analytic extension of $\phi$ to $D_2$.  Thus, $\phi$ extends analytically to a function $\tilde{\phi}$ analytic on $\D\cup D_2$. Because $\tilde{\phi}(D_2) = f^{-1}(h(D_2))$ is contained in the range of $f^{-1}$ which is, in turn, contained in $D_1\subseteq \D$, we see that $\tilde{\phi}$ maps $D_2$ into $\D$.  Hence $\tilde{\phi}$ maps $\D\cup D_2$ into $\D$.

  Let $f\in X$ be arbitrary.  Then $f\circ \phi$ has analytic extension $f\circ \tilde{\phi}$ to $\D\cup D_2$.    Because $C_\phi$ is invertible on $X$, we conclude that every function in $X$ has analytic extension to $\D\cup D_2$, contrary to hypothesis (iv).  This contradiction tells us that $\phi$ must have radial limit of modulus $1$ a.e.\ on $\partial \D$; that is, $\phi$ is an inner function.  Since univalent inner functions must be automorphisms (see, e.g., \cite[Corollary 3.28]{CMB}), our proof is complete.
\end{proof}  

With somewhat stronger hypotheses on the set $X$, we  obtain a version of Theorem~\ref{MT1} applying to weighted composition operators.

\begin{theorem}\label{MT2}
  Suppose that $X$ is a set of functions analytic on $\D$ such that (i)   $W_{\psi, \phi}$ maps $X$ to $X$, (ii)  $X$ contains a nonzero constant function, (iii) $X$ contains a function of the form $z\mapsto z+ c$ for some constant $c$, (iv)   there is a dense subset $S$ of the unit circle such that for each point in $S$ there is function in $X$ that does not extend analytically to a neighborhood of that point.    If $W_{\psi, \phi}:  X\rightarrow X$ is  invertible, then $\phi$ is an automorphism of $\D$.  
\end{theorem}
\begin{proof}      Let $W_{\psi, \phi}:  X\rightarrow X$  be invertible. Let $c$ be constant such that $g(z) = z +c$ belongs to $X$ and let $c_1\ne 0$ be a constant such that $h(z) = c_1$ belongs to $X$.       Because $W_{\psi , \phi}$ is invertible and $h\in X$,  there is a function $f_1\in X$ such that $\W f_1 = h$, from which it follows that $\psi(z) = c_1/f_1(\phi(z))$ for each $z\in \D$.  Also, there is a function $f_2\in X$ such that $\W f_2 = g$ so that $\psi(z)f_2(\phi(z)) = z + c$ for all $z\in \D$ or
$$
c_1\frac{f_2}{f_1}(\phi(z))- c = z
$$
for each $z\in \D$.   It follows that $\phi$ must be univalent on $\D$.  Moreover,  $\phi^{-1}$ has a meromorphic extension  $q:=c_1\frac{f_2}{f_1} - c$ from $\phi(\D)$ to $\D$.

 Suppose, in order to obtain a contradiction that $\phi$ has radial limit of modulus less than $1$ on a subset $E$ of $\partial D$ having positive measure.  Because $E$ has positive measure,  there is a positive number $t$ less than $1$ such that  measure of the set $T:= \{\zeta\in \partial \D:  |\phi(\zeta)| < t\}$ is also positive.   Because $\phi$ is nonconstant, the set $\phi(T)$ cannot be finite. Because $q(\phi(r\zeta)) = r\zeta$ for each $r\in [0 ,1)$, we see that $\phi(\zeta)$ cannot be a pole of $q$ for any $\zeta\in T$; in fact $|q(\phi(\zeta))| = 1$ for $\zeta\in T$ assures us that no cluster point of $\phi(T)$ is a pole of $q$.  Because $q$ is not constant, there is a point $\phi(\zeta_0)$  of $\phi(T)$ at which $q$ has nonzero derivative. Thus there is a an open disk $D_0$ centered at $\phi(\zeta_0)$  and contained in $\D$ such that $q$ is invertible on $D_0$---that is, there is  an analytic function $q^{-1}$ on $q(D_0)$ such that $q^{-1}\circ q$ is the identity on $D_0$.    Since $q(\phi(\zeta_0)) = \zeta_0$, we see $q$ maps  $D_0$ to an  open set containing an open disk $D_1$ containing $\zeta_0$.   Since $q^{-1}(r\zeta_0) = \phi(r\zeta_0)$ for every $r$ such that $r\zeta_0\in D_1$, we see that $q^{-1}|_{D_1}$ is an analytic extension of $\phi$ from $\D\cap D_1$ to $D_1$. Thus $\phi$ has an analytic extension $\tilde{\phi}$ from $\D$ to  $\D\cup D_1$.
 
Recall that $\psi(z) = \frac{c_1}{f_1(\phi(z))}$ for each $z\in \D$.  Note that  $\tilde{\phi}(D_1)\subseteq \D$  because $ \tilde{\phi}(D_1) = q^{-1}(D_1)\subseteq D_0\subseteq \D$.  Thus,  the function $f_1\circ \tilde{\phi}$ is analytic on $\D\cup D_1$ and thus $\psi$, which is analytic on $\D$, has meromorphic extension  to $\D\cup D_1$. We choose a point $\zeta_1\in D_1\cap \partial \D$ at which this meromorphic extension of $\psi$ is analytic and a disk $D_2$ centered at $\zeta_1$ and contained in $D_1$ such that $\psi$ has analytic extension $\tilde{\psi}$ defined on $\D\cup D_2$.  Now observe that since $\tilde{\phi}$ maps $D_2$ into $\D$,  for any $f\in X$, $f\circ \phi$ extends to  be analytic on $\D\cup D_2$, agreeing with $f\circ \tilde{\phi}$. Hence for any $f\in X$, the product $\psi\cdot (f\circ \phi)$ has analytic extension $\tilde{\psi} f\circ\tilde{\phi}$ to $D\cup D_2$.  However, since $W_{\psi,\phi}$ is invertible on $X$, this yields that every function in $X$ extends analytically to $\D\cup D_2$, contrary to our hypothesis (iv).    From this contradiction,  it  follows that $\phi$ must have radial limit of modulus $1$ almost everywhere on $\partial \D$.  That is, $\phi$ is inner and hence must be an automorphism because it is univalent.  
\end{proof}  

If the set $X$ of analytic functions on $\D$ is automorphism invariant, that is, $f\circ \phi\in X$ whenever $f\in X$ and $\phi$ is an automorphism of $\D$, then Theorems~\ref{MT1} and \ref{MT2} characterize invertibility:  
\begin{corollary} If $X$ and $\phi$ satisfy the hypotheses of Theorem~\ref{MT1} and $X$ is  automorphism invariant, then $C_\phi$ is invertible on $X$ iff $\phi$ is an automorphism of $\D$. If $X$, $\psi$, and $\phi$ satisfy the hypotheses of Theorem~\ref{MT2} and $X$ is  automorphism invariant, then $W_{\psi, \phi}$ is invertible on $X$ iff $\phi$ is an automorphism of $\D$ and $\psi$ as well as $1/\psi$ are multipliers of $X$.
\end{corollary}

 Recall that a function $g$ is a multiplier of a set $X$ provided that $gf\in X$ whenever $f\in X$.  Note that if $\phi$ is an automorphism and $X$ is an automorphism-invariant set, then $1/\psi$ is a multiplier of $X$ iff $1/\psi\circ\phi^{-1}$ is a multiplier.  It is easy to see that if $W_{\psi,\phi}$ is invertible on $X$ and the hypotheses of Theorem~\ref{MT2} hold, then $W_{\psi,\phi}^{-1} = W_{1/\psi\circ\phi^{-1}, \phi^{-1}}$.

\section{applications}

 Theorems~\ref{MT1} and \ref{MT2} of the previous section are widely applicable, yielding both old and new results.  For example, for $0 < p < \infty$,  they may be applied to composition operators and weighted composition operators  on the Hardy and Bergman spaces $H^p(\D)$ and $A^p(\D)$, where, at least for composition operators, the characterization of invertibility is well known  (see, e.g., \cite[Exercise 2.1.15 \& Theorem 1.6]{CMB}). 
  However our theorems also may be applied to many other spaces for which invertibility results are not in the literature; for instance,   the space $S^p$,  the Bloch space $\mathcal{B}$, the disk algebra, as well as the  Lipschitz spaces Lip$_\alpha(\D)$ ($0 < \alpha \le 1$).  (Definitions of all these function spaces may be found, e.g., in \cite[Chapter 4]{CMB}.)  Theorems~\ref{MT1} and \ref{MT2} also yield, with no restrictions on the weight sequence $\beta$, that invertibility of $C_\phi$ or $W_{\psi, \phi}$ on $H^2(\beta)$ implies $\phi$ is an automorphism of $\D$---see Theorem~\ref{WHST} below.    Before turning to this weighted-Hardy-space result, we record explicitly the consequences of our work for the spaces $S^p$ (on which composition operators are studied, in, e.g., \cite{BMS}, \cite{tt},  and \cite{JS}).
  
  Let $H(\D)$ be the collection of all analytic functions of $\D$ and for $0 < p < \infty$, let $H^p(\D)$ be the Hardy-$p$ space of $\D$, which consists of all $f\in H(\D)$ satisfying
\begin{equation}\label{HSD}
\|f\|_p^p : =  \sup_{0\le r < 1} \frac{1}{2\pi} \int_{\partial \D} |f(re^{it})|^p \, dt  < \infty.
 \end{equation}
  Recall that $S^p = \{f\in H(\D): f' \in H^p(\D)\}$.  Note that  for every $p$, the space $S^p$ contains $z\mapsto z$ as well as the constant function $z\mapsto 1$; moreover, $S^p$ contains antiderivatives of bounded analytic functions on $\D$.  Thus, there are functions in $S^p$ that don't extend analytically across any point of the unit circle (consider an antiderivative of a Blaschke product whose zero sequence accumulates at each point of $\partial \D$).  Thus, Theorems \ref{MT1} and  \ref{MT2} may be applied to $S^p$.  In fact, because $S^p$ is easily seen to be automorphism invariant, we obtain the following complete characterization of invertibility .

  \begin{theorem}  The composition operator $C_\phi$ is invertible on $S^p$ if and only if $\phi$ is an automorphism of $\D$.  Moreover, the weighted composition operator $W_{\psi, \phi}$ is invertible on $S_p$ if and only if $\phi$ is an automorphism of $\D$ and both $\psi$ and $1/\psi$ are multipliers of $S^p$.
  \end{theorem}
  
   We turn now to weighted Hardy spaces.    Because any weighed Hardy space $H^2(\beta)$ contains the function $z\mapsto z$ as well as the constant function $z\mapsto 1$, the only issue to consider in attempting to apply either of the theorems of the preceding section is that of existence of functions in $H^2(\beta)$ that don't extend analytically to neighborhoods of points on the unit circle.   Because $H^2(\beta)$ is rotation invariant, if there is a function that does not extend analytically to a neighborhood of some point $\zeta$ on the unit circle,  the same will be true of every point on the unit circle.  Consider,
$$
f(z) = \sum_{j=1}^\infty \frac{z^j}{j\beta(j)},
$$
which belongs to $H^2(\beta)$.  Suppose that for each point $\zeta\in\partial \D$, this function $f$ analytically extends to a neighborhood of $\zeta$, then $f$ would be analytic on a disk having radius larger than $1$, making $\limsup (j\beta(j))^{-1/j} < 1$, which implies $\liminf (\beta(j))^{1/j} > 1$.    Thus if $\liminf \beta(j)^{1/j} = 1$, then $f$ must fail to extend analytically to a neighborhood of some point on $\partial \D$.   Keeping in mind that $H^2(\beta)$ is rotation invariant, we see that if  $\liminf \beta(j)^{1/j} = 1$,  then $X = H^2(\beta)$ satisfies condtion (iv) of Theorems~\ref{MT1} and \ref{MT2}.    

Thus we have
\begin{corollary} \label{BC} Suppose that $H^2(\beta)$ is a weighed Hardy space such that  
\begin{equation}
\liminf \beta(j)^{1/j}=1
\end{equation}
and that $W_{\psi,\phi}$ takes $H^2(\beta)$ into itself.  If  $W_{\psi, \phi}$ is invertible on  $H^2(\beta)$; then $\phi$ must be an automorphism of $\D$.
\end{corollary}

Choosing $\psi \equiv 1$ in the preceding corollary, we see that only automorphisms can induce invertible composition operators on weighted Hardy spaces for which $\liminf \beta(j)^{1/j}=1
$ holds.   Previously obtained conditions on $H^2(\beta)$, such as \cite[Theorems 1.6 and 2.15]{CMB}, \cite[Theorem 2]{MP}, and \cite[Theorem 2.0.1]{GGPP},  have required conditions like that of  \cite[Theorem 2]{MP}:
$$
\sum_{n=0}^\infty \frac{n^{2k}}{\beta(n)^2} = \infty \quad \text{for some}\ k \ge 0,
$$
which implies $\liminf \beta(j)^{1/j} = 1$ (in view of the fact that $\liminf \beta(j)^{1/j} \ge 1$ for any weighted Hardy space).

If the weighted Hardy space $H^2(\beta)$ of Corollary~\ref{BC} is automorphism invariant, then Corollary~\ref{BC} immediately yields a characterization of invertibility of composition operators and weighted composition operators, which is stated below as Theorem~\ref{IAI}.  Conditions assuring that $H^2(\beta)$ is automorphism invariant may be found in, e.g., \cite[Theorem 3.3]{MSS}.  For instance, $\beta(j) = (j+1)^\alpha$ for some real number $\alpha$ is sufficient to ensure that $H^2(\beta)$ is automorphism invariant. 

\begin{theorem}\label{IAI}   Suppose that $H^2(\beta)$ is an automorphism invariant weighted Hardy space such that $\liminf \beta(j)^{1/j}=1$.  Then $C_\phi: H^2(\beta)\rightarrow H^2(\beta)$ is invertible iff $\phi$ is an automorphisms of $\D$ and $W_{\psi, \phi}:H^2(\beta)\rightarrow H^2(\beta)$ is invertible iff $\phi$ is an automorphism and both $\psi$ and $1/\psi$ are multipliers of $H^2(\beta)$.  
\end{theorem}

Multipliers of weighted Hardy spaces are necessarily bounded analytic functions on $\D$.    It's easy to check  that for $w\in \D$, the function $K_w(z) = \sum_{j=0}^\infty (\bar{w}z)^j/\beta(j)^2$ belongs to $H^2(\beta)$ and is the reproducing kernel at $w$ for $H^2(\beta)$:
$$
\langle f, K_w\rangle = f(w) \quad \text{for each}\ f \in H^2(\beta).
$$
Thus (norm) convergence of a sequence of functions in $H^2(\beta)$ yields pointwise convergence on $\D$.  Applying the closed graph theorem, we see that if $h$ is a multiplier of $H^2(\beta)$, then the multiplication operator $M_h$ on $H^2(\beta)$, defined by $M_h f = hf$, is bounded on $H^2(\beta)$.   Now observe that for each $w\in \D$, $\overline{h(w)}$ is an eigenvalue with corresponding eigenfunction $K_w$ for $M_h^*: H^2(\beta)\rightarrow H^2(\beta)$.  Hence, the image of $\D$ under $\bar{h}$ is in the spectrum of $M_h$ and hence if $h$ is a multiplier of $H^2(\beta)$, then $h$ is a bounded analytic function on $\D$ with $\sup\{|h(z)|: z\in \D\}\le \|M_h\|$.    Conversely, integral representations of norms for some of the weighted Hardy spaces make it clear that every bounded analytic function must be a multiplier of the space in question.   Consider the spaces $H^2(\beta_{
\alpha})$ of $\D$, $\alpha \ge -1$, where $\beta_\alpha(j)^2 = (j+1)^{-1-\alpha}$ for all $j\ge0$.  $H^2(\beta_{-1})$ is the classical Hardy space $H^2(\D)$ whose norm is given by (\ref{HSD}) with $p =2$.  The spaces $H^2(\beta_\alpha)$ with $\alpha > -1$ are standard-weight Bergman spaces having equivalent norm
$$
\|f\|_{H^2(\beta_\alpha)}^2 = \frac{1}{\pi} \int_0^{2\pi}\int_0^1 |f(re^{i\theta})|^2(1-r^2)^\alpha r\, dr\, d\theta
$$
(see, \cite{TM}).  The integral forms for the norms on the spaces $H^2(\beta_{\alpha})$, $\alpha \ge -1$,  make is clear that any bounded analytic function on $\D$ is a multiplier of the spaces.  Thus we obtain the following (cf.\ \cite[Theorem 2.0.1]{GGPP}) as a corollary of our work.

\begin{theorem} The weighted composition operator $W_{\psi, \phi}$ is invertible on $H^2(\beta_\alpha)$ for some $\alpha \ge -1$, iff $\phi$ is an automorphism of $\D$ and $\psi$ is both bounded and bounded away from $0$ on $\D$.  
\end{theorem}

What happens if $\liminf \beta(j)^{1/j} > 1$? Here, again, it turns out that invertibility of $W_{\psi,\phi}$ implies $\phi$ is  an automorphism---in fact a rotation $z\mapsto \zeta z$ for some $\zeta\in \partial \D$.   

\begin{theorem}\label{WHST}  Suppose that $W_{\psi,\phi}$ is a bounded, invertible operator on the weighted Hardy space $H^2(\beta)$; then $\phi$ must be an automorphism of $\D$. Moreover, if \linebreak $\liminf \beta(j)^{1/j} > 1$, then $\phi$ must be a rotation automorphism.
\end{theorem}
\begin{proof}   We have already shown the result is valid when $\liminf \beta(j)^{1/j} = 1$ (Corollary~\ref{BC}).   We now consider the case where
$\liminf \beta(j)^{1/j} = t > 1$. In this case, it's easy to see that every function in $H^2(\beta)$ extends to be analytic on the disk $t\D$ and that 
$$
f(z) = \sum_{j=0}^\infty \frac{z^j}{j\beta(j)}
$$
belongs to $H^2(\beta)$, but does not extend to be analytic on a disk of radius larger than $t$.  Thus,  there is a point $\mu$ in $\{z: |z| = t\}$  such that $f$  fails to have analytic extension to a neighborhood of $\mu$.    Let  $X = H^2(\beta)$ and note that because $z\mapsto z$ as well as $z\mapsto 1$ are in $X$ and $X$ is rotation invariant, hypotheses (ii)--(iv)  of Theorem~\ref{MT2} are satisfied with $t\D$ replacing $\D$ and we can take $S = \{z: |z| = t\}$.    We will show hypothesis (i) is valid too for appropriate extensions of $\psi$ and $\phi$ to $t\D$.  Specifically,  we will show  $\psi$ has an analytic extension $\tilde{\psi}$ to $t\D$ and $\phi$ has an analytic extension $\tilde{\phi}$ to $t\D$ that is a selfmap of $t\D$. Because $W_{\psi,\phi}$ is an invertible operator on $H^2(\beta)$, it follows that $W_{\tilde{\psi}, \tilde{\phi}}$ is an invertible mapping on the set $X$. 

 By hypothesis,  $\phi$ is a selfmap of $\D$ and $\psi$ is an analytic function on $\D$  such that $W_{\psi,\phi}$ is invertible on $H^2(\beta)$. Since $z\mapsto 1$ is in $H^2(\beta)$ and $H^2(\beta)$ is invariant under $W_{\psi,\phi}$, then $\psi = W_{\psi,\phi} 1$ is in $H^2(\beta)$, which means $\psi$ must extend analytically to $t\D$. We use $\tilde{\psi}$ to denote this extension.    Because $W_{\psi,\phi}$ is invertible on $H^2(\beta)$, there is a function $g\in H^2(\beta)$ such that $\psi\cdot g\circ \phi = 1$ on $\D$, which means $\psi$ is nonzero on $\D$ so that $\tilde{\psi}$ is certainly not the zero function on $t\D$. Because $z\mapsto z$ is in $H^2(\beta)$, we see that $\psi \phi$ is in $H^2(\beta)$ and extends to a function $\tilde{q}$ analytic on $t\D$, and thus $\phi$ extends to a meromorphic function $\tilde{q}/\tilde{\psi}$ on $t\D$. However, we claim  $\tilde{q}/\tilde{\psi}$ cannot have any poles, so that $\phi$ has analytic extension $\tilde{\phi}$ to $t\D$.  If  $\tilde{q}/\tilde{\psi}$ had poles in $t\D$, then for sufficiently large positive integers $n$, the function $\tilde{\psi}\cdot(\tilde{q}/\tilde{\psi})^n$ would have poles.  But  $\tilde{\psi}\cdot(\tilde{q}/\tilde{\psi})^n$ agrees with  $\psi \phi^n= W_{\psi,\phi}z^n$ on $\D$, and because $z\mapsto z^n$ is in $H^2(\beta)$, we see that  $\psi \phi^n$ has analytic extension to $t\D$. This  extension must agree with $\tilde{\psi}\cdot(\tilde{q}/\tilde{\psi})^n$ except at any singularities of $\tilde{\psi}\cdot(\tilde{q}/\tilde{\psi})^n$ and thus those singularities must be removable.  We have proved our claim:  $\phi$ has analytic extension $\tilde{\phi}$ to $t\D$.    
 
Suppose, in order to obtain a contradiction that at some point $w_0$ of $t\D$, $|\tilde{\phi}(w_0)| > t$.  For each $x\in \R$, consider an uncountable collection of curves $\{\gamma_x\}_{x\in \R}$ such that 
\begin{itemize}
\item[(i)] $\gamma_x: [0,1]\rightarrow t\D$ has endpoints $0$ and $w_0$ and 
\item[(ii)] $\gamma_x([0,1])\cap \gamma_y([0,1]) = \{0, w_0\}$ when $x\ne y$.  
\end{itemize}
For each $x\in \R$, let $r_x = \inf\{r\in [0,1]: |\tilde{\phi}(\gamma_x(r))| \ge t\}$, and note that because $\tilde{\phi}\circ \gamma_x$ is continuous on $[0,1]$, $|\tilde{\phi}(\gamma_x(r_x))| = t$.  For each $x\in \R$, let $\gamma_x(r_x) = z_x$, a point in $t\D$.  Thus we have an uncountable collection $\{z_x:x\in \R\}\subseteq t\D$ such that the image of each under $\tilde{\phi}$ has modulus $t$.  Because on $t\D$,  $\tilde{\psi}$ is not the zero function  and $\tilde{\phi}$ is not constant (it agrees with the selfmap $\phi$ of $\D$ and takes values outside $\D$), there must be a real number $*$ such that $z_*$ is a point at which $\tilde{\phi}$ has nonzero derivative and $\tilde{\psi}$ is nonzero.  Because $\tilde{\phi}$ has nonzero derivative at $z_*$, it is invertible on a neighborhood of $z_*$ with inverse $\tilde{\phi}^{-1}$. Thus,  there is an open disk $D_1$ about $\tilde{\phi}(z_*)$ such that $\tilde{\phi}^{-1}$ is defined on $D_1$ and maps $D_1$ to a neighborhood of $z_*$ contained in $t\D$ on which $\tilde{\psi}$ is nonzero.   Let $g$ be a function in $H^2(\beta)$ that fails to have analytic extension to any neighborhood of the point $\tilde{\phi}(z_*)$, which is on $\partial(t\D)$.  Of course $g$ has analytic extension $\tilde{g}$ to $t\D$, so that $\tilde{g}$ fails to extend analytically to a neighborhood of $\tilde{\phi}(z_*)$.   We know that $\psi \cdot g\circ \phi = h$ where $h$ is in $H^2(\beta)$. Thus $h$ has an analytic extension $\tilde{h}$  to $t\D$ that must agree with $\tilde{\psi} \cdot \tilde{g}\circ\tilde{\phi}$ on $\gamma_*[0,r_*)$.  Note that $\displaystyle{\tilde{g} = \frac{1}{\vstrut \tilde{\psi}\circ \tilde{\phi}^{-1}}\tilde{h}\circ \tilde{\phi}^{-1}}$ on $D_1\cap \tilde{\phi}(\gamma_*[0,r_*))$ and thus that $\displaystyle{\frac{1}{\vstrut \tilde{\psi}\circ \tilde{\phi}^{-1}}\tilde{h}\circ \tilde{\phi}^{-1}}$ is an analytic extension of $\tilde{g}$ to $D_1$, a contradiction.  We conclude that $\tilde{\phi}$ maps $t\D$ into $t\D$; that is,  $\tilde{\phi}$ is a analytic selfmap of $t\D$.

All the hypotheses of Theorem~\ref{MT2} are satisfied for the operator $W_{\tilde{\psi}, \tilde{\phi}}$ with $\D$ replaced by $t\D$ and $X = H^2(\beta)$.  Because the proof of Theorem~\ref{MT2} is not dependent on the radius of the disk in question, we see that $\tilde{\phi}$ must be an automorphism of $t\D$.  However, $\tilde{\phi}$, which agrees with $\phi$ on $\D$, must also be a selfmap of $\D$.   No non-rotational automorphism of $t\D$ can be a selfmap of $\D$:  An automorphism of $t\D$ must take the form 
$$
A(z) = t\zeta \frac{z - tp}{t-\bar{p}z}
$$
for some $p\in \D$ and $\zeta\in \partial \D$. If $A$ is a selfmap of $\D$, then $|A(-p/|p|)| = (t + t^2|p|)/(t+|p|) \le 1 $, making $p = 0$ and $A$ a rotation. Thus $\tilde{\phi}$ and hence $\phi$ takes the form $z\mapsto \zeta z$ for some $\zeta\in \partial \D$.  Hence, in particular, $\phi$ is a automorphism of $\D$, as desired.  

To complete the proof of the theorem, we must consider the case where
$$
\liminf \beta(j)^{1/j} = \infty.
$$
Here all functions in $H^2(\beta)$ extend to be entire functions.  For this argument, we will not introduce new notation for extensions.    Since $z\mapsto 1$ belongs to $H^2(\beta)$, $\psi = W_{\psi,\phi}1$ is in $H^2(\beta)$ and hence $\psi$ represents an entire function.      Also, because $z\mapsto z^n$ belongs to $H^2(\beta)$ for every positive integer $n$, $\psi\phi^n$ is entire for every $n$ and it follows that $\phi$ must be entire.  Because    $W_{\psi, \phi}$ is invertible,  there must be a $g\in H^2(\beta)$ such that $\psi\cdot g\circ \phi = 1$ on $\D$ and hence on $\C$.  Thus $\psi$ is nonzero on $\C$.  Also, there must be a function $q\in H^2(\beta)$ such that $z = \psi \cdot q\circ\phi = (q/g)\circ \phi$ on $\C$.  It follows that $\phi$ is univalent on $\C$.  Because $\phi$ is univalent and entire, $\phi(z) = az + b$ for some constants $a$ and $b$ with $a\ne 0$.  Because $\phi$ is selfmap of $\D$, $|a| + |b| \le 1$.   We show that $b = 0$ and $|a| =1$ to complete the proof. 

 Suppose in order to obtain a contradiction that $b\ne 0$.  Then $|a| <1$.  Because $W_{\psi, \phi}$ is bounded and invertible, its inverse is bounded. It's easy to see that 
$$
\W^{-1} f = \frac{1}{\psi\circ\phi^{-1}} f\circ \phi^{-1}
$$
where $\phi^{-1}(z) = z/a - b/a$.  We obtain the contradiction that $\W^{-1}$ is unbounded on $H^2(\beta)$ (assuming $|a| <1$),   Suppose that $\nu:=1/\psi\circ\phi^{-1}$ is not constant. Then by Liouville's Theorem, there is a sequence $(c_n)$ in $\C$ such that $|c_n|\rightarrow \infty$ and $|\nu(c_n)|\rightarrow \infty$.  Let $k(z) = \sum_{j=0}^\infty z^j/\beta(j)^2$ be the generating function for $H^2(\beta)$, which is entire for our situation.  For any $w\in \C$, $K_w(z) :=k(\bar{w}z)$ belongs to $H^2(\beta)$ and the square of its norm is
$k(|w|^2)$.  Also for all $w\in \C$ and $g\in H^2(\beta)$, we have $\langle g, K_w\rangle =g(w)$.  Note that because $\phi^{-1}(c_n)/c_n$ approaches $1/a$ as $n\rightarrow \infty$,  we have $|\phi^{-1}(c_n)| >|c_n|$ for $n$ sufficiently large.  This means, because $k$ is increasing along the positive real axis, that $\|K_{\phi^{-1}(c_n)}\|\ge \|K_{c_n}\|$ for $n$ sufficiently large; let's say for $n\ge J$.  Consider the unit vector $k_w:=K_w/\|K_w\|$.  Since $W_{\psi, \phi}^{-1}$ is a bounded operator on $H^2(\beta)$ its adjoint is also bounded, so that there is a constant $M$ such that 
$$
M\ge\left\| \left(W_{\psi, \phi}^{-1}\right)^* k_w \right\|
$$
for all $w$.  Choosing $w = c_n$, and considering that  $\left(W_{\psi, \phi}^{-1}\right)^* k_w  = \overline{\nu(w)}K_{\phi^{-1}(w)}/\|K_w\|$, the preceding inequality yields, for $n\ge J$, 
$$
M\ge |\nu(c_n)| \frac{\|K_{\phi^{-1}(c_n)}\|}{\|K_{c_n}\| } = |\nu(c_n)| \frac{k(|\phi^{-1}(c_n)|^2)}{k(|c_n|^2)} \ge |\nu(c_n)|,  
$$
 contradicting $\lim |\nu(c_n)|  = \infty$.

It follows that if $\W^{-1}$ is to be bounded. then $\nu = 1/\psi\circ\phi^{-1}$ must be constant.  Thus $\psi$ is constant and thus $W_{\psi, \phi}$ is simply a constant multiple of a composition operator.  Let $\psi = \alpha$, a constant.  We are assuming that $\alpha C_{\phi}$ is bounded and invertible on $H^2(\beta)$, where $\phi(z) = az + b$, $|a| + |b| \le 1$, and $b\ne 0$ so that $|a| <1$.  Because $\alpha C_\phi$ is bounded and invertible its inverse $(1/\alpha)C_{\phi^{-1}}$ is as well.  We apply this inverse to the unit vector $h(z): = z^n/\beta(n)$ and note that the $n$-th Maclaurin coefficient of $(1/\alpha)h\circ\phi^{-1}$ is $(1/\alpha)(1/a)^n(1/\beta(n))$; thus
$$
\|(1/\alpha)C_{\phi^{-1}} h\|^2 \ge \frac{1}{|\alpha|^2}\frac{1}{|a|^{2n}},
$$
and the quantity on the right of the preceding inequality goes to infinity as $n\rightarrow \infty$.  Thus our assumption that $b\ne 0$ has led to a  contradiction and we see that $\phi(z) =az$ for some $a$ such that $|a|\le 1$.  If $|a|<1$, then the argument just completed again shows that $W_{\psi,\phi}$ cannot have a bounded inverse. Thus $|a| = 1$ and $\phi$ is rotation automorphism, completing the proof of the theorem.
\end{proof}

Because rotations always induce bounded composition operators on $H^2(\beta)$, we conclude that if $H^2(\beta)$ is a weighted Hardy space for which $\liminf \beta(j)^{1/j} > 1$, and $\phi$ is a selfmap of $\D$, then $C_\phi$ is invertible iff $\phi(z) = \zeta z$ for some $\zeta\in \partial \D$. Similarly, we obtain that $W_{\psi,\phi}$ is invertible iff $\phi(z) = \zeta z$ for some $\zeta\in \partial \D$ and $\psi$ as well as $1/\psi$ are multipliers of $H^2(\beta)$.

\bibliographystyle{amsplain}

\end{document}